\documentclass{article} 
\usepackage{nips10submit_e,times}

\nipsfinalcopy

\usepackage{amssymb,amsfonts,amsmath,amsthm,amscd,dsfont,mathrsfs}
\usepackage{graphicx,float,subfig,psfrag,epsfig,color}

\newtheorem{theorem}{Theorem}[section]
\newtheorem{lemma}[theorem]{Lemma}
\newtheorem{corollary}[theorem]{Corollary}

\newtheorem{proposition}[theorem]{Proposition}
\newenvironment{proofof}[1]{{\bf Proof of #1:}}{$\qed$\par}

\newcommand{\comment}[1]{}

\def\III{|\!|\!|}

\def\prob{{\mathbb P}}

\newcommand{\reals}{{\mathds R }}

\newcommand{\naturals}{{ \mathds N}}

\def\cE{{\cal E}}
\def\E{{\mathbb E}}

\def\prob{{\mathbb P}}

\def\eps{\epsilon}

\def\normal{{\sf N}}
\def\de{{\rm d}}
\def\dpi{{\partial_+i}}
\def\complex{{\mathbb C}}
\def\cL{{\mathcal L}}
\def\cG{{\cal G}}
\def\cE{{\cal E}}
\def\cV{{\cal V}}
\def\dg{{\rm deg}}
\def\sign{{\rm sign}}
\def\hG{\widehat{G}}
\def\hQ{\widehat{Q}}

\title{Learning Networks of\\  Stochastic Differential Equations}

\author{
Jos\'e Bento\\
Department of Electrical Engineering\\
Stanford University\\
Stanford, CA 94305\\
\texttt{jbento@stanford.edu} \\
\And
Morteza Ibrahimi\\
Department of Electrical Engineering\\
Stanford University\\
Stanford, CA 94305\\
\texttt{ibrahimi@stanford.edu} \\
\AND
Andrea Montanari\\
Department of Electrical Engineering and Statistics\\
Stanford University\\
Stanford, CA 94305\\
\texttt{montanari@stanford.edu} \\
}

\begin{document}

\maketitle

\begin{abstract}
We consider linear models for stochastic dynamics.
To any such model can be associated a 
network (namely a directed graph) describing which degrees of freedom
interact under the dynamics. We tackle the problem of learning 
such a network from observation of the system trajectory over a time 
interval $T$.

We analyze the $\ell_1$-regularized least squares algorithm and,
in the setting in which the underlying
network is sparse, we prove performance guarantees that are 
\emph{uniform in
the sampling rate} as long as this is sufficiently high.
This result substantiates the notion of a well defined
`time complexity' for the network inference problem.
\end{abstract}

\textbf{keywords:}
Gaussian processes, model selection and structure learning, graphical models, sparsity and feature selection.
%
\section{Introduction and main results}

Let $G= (V,E)$ be a directed graph with weight $A^{0}_{ij}\in\reals$
associated to the directed edge $(j,i)$ from $j\in V$ to $i\in V$. 
To each node $i\in V$ in this network is associated an independent
standard Brownian motion $b_i$ and a variable $x_i$ taking values in $\reals$
and evolving according to  
\begin{eqnarray*}
\de x_i(t) = \sum_{j\in \dpi} 
A^0_{ij} x_j(t)\, \de t +\, \de b_i(t) \, ,
\end{eqnarray*}
where $\dpi = \{j\in V: \, (j,i)\in E\}$ is the set of `parents'
of $i$. 
Without loss of generality we shall take $V=[p]\equiv\{1,\dots,p\}$.
In words, the rate of change of $x_i$ is given 
by a weighted sum of the current values of its neighbors, 
corrupted by white noise. 
In matrix notation, the same system is then represented by
\begin{eqnarray}
\de x(t) = 
A^0 x(t)\, \de t + \, \de b(t) \, , \label{eq:BasicModel}
\end{eqnarray}
with $x(t)\in\reals^p$, $b(t)$ a $p$-dimensional standard Brownian motion
and $A^0\in\reals^{p\times p}$ a matrix with entries 
$\{A^0_{ij}\}_{i,j\in [p]}$ whose sparsity pattern 
is given by the graph $G$. We assume that the linear system
$\dot{x}(t) = A^0 x(t)$ is stable (i.e. that the spectrum of $A^0$ 
is contained in $\{z\in\complex\, :\, {\rm Re}(z)<0\}$).
Further,  we assume that $x(t=0)$ is in its stationary state. 
More precisely, $x(0)$ is a Gaussian random variable independent of $b(t)$,
distributed according to the invariant measure. 
Under the stability assumption, this a mild restriction, since the 
system converges exponentially to stationarity.

A portion of time length $T$
of the system trajectory $\{x(t)\}_{t\in [0,T]}$ is observed
and we ask under which conditions these data are sufficient to reconstruct
the graph $G$ (i.e., the sparsity pattern of $A^0$). We are particularly 
interested in computationally efficient procedures, and
in characterizing the scaling of the learning time for large networks.
Can the network structure be learnt in 
a time scaling linearly with the number of its degrees of freedom?

As an example application, chemical reactions can be conveniently modeled
by systems of non-linear stochastic differential equations,
whose variables encode the densities of various chemical species
\cite{Gillespie,Higham}. Complex biological networks might involve
hundreds of such species \cite{BioBook}, and learning stochastic
models from data is an important (and challenging)
computational task \cite{ToniEtAl}.
Considering one such chemical reaction network in proximity
of an equilibrium point, the model (\ref{eq:BasicModel}) can be used 
to trace fluctuations of the species counts with respect
to the equilibrium values. The network $G$ would represent 
in this case the interactions between different chemical factors.
Work in this area focused so-far on low-dimensional networks, i.e. 
on methods that are guaranteed to be correct for fixed $p$, as $T\to\infty$,
while we will tackle here the regime in which both $p$ and $T$ diverge.

Before stating our results, it is useful to stress a few 
important differences with respect to classical graphical model
learning problems:
\begin{itemize}
\item[$(i)$] Samples are not independent.
 This can (and does) increase the sample complexity.
\item[$(ii)$] On the other hand,  infinitely many samples are
given as data (in fact a collection indexed by the continuous parameter
$t\in [0,T]$).  Of course one can select a finite subsample,
for instance at regularly spaced times $\{x(i\,\eta)\}_{i=0,1,\dots}$.
This raises the question as to whether the learning performances
depend on the choice of the spacing $\eta$. 
\item[$(iii)$] In particular, one expects that choosing $\eta$ 
sufficiently large as to make the configurations in the subsample approximately
independent can be harmful.
Indeed, the matrix $A^0$ contains more information than the stationary
distribution of the above process (\ref{eq:BasicModel}), and only
the latter can be learned from independent samples.
\item[$(iv)$] On the other hand, letting $\eta\to 0$, one can produce
an arbitrarily large number of distinct samples. However,
samples become more dependent, and intuitively one expects that 
there is limited information to be harnessed from a given time interval $T$.
\end{itemize}
Our results confirm in a detailed and quantitative way these intuitions.
%
%
\subsection{Results: Regularized least squares}
\label{sec:ResContinuumRegularized}

Regularized least squares is an efficient and well-studied
method for support recovery. We will discuss relations 
with existing literature in Section \ref{sec:Related}.

In the present case, the algorithm reconstructs independently each
row of the matrix $A^0$.
The $r^{{\rm th}}$ row, $A^0_r$, is estimated by solving the following 
convex optimization problem for $A_r\in\reals^p$
\begin{equation}
{\rm minimize}\;\;\;
 \cL(A_r;\{x(t)\}_{t\in [0,T]}) + \lambda \| A_r\|_1\, ,\label{eq:cont_reg_prob}
\end{equation}
where the likelihood function $\cL$ is defined by
\begin{equation}
\cL(A_r;\{x(t)\}_{t\in [0,T]}) = 
 \frac{1}{2T} \int^T_0\!\! (A_r^* x(t))^2\,\, \de t - \frac{1}{T}\int^T_0\!\!
(A_r^* x(t))\,\,\, \de x_r(t)\, .\label{eq:ContCost}
\end{equation}
(Here and below $M^*$ denotes the transpose of matrix/vector $M$.)
To see that this likelihood function is indeed related to least squares,
one can \emph{formally} write $\dot{x}_r(t) =\de x_r(t)/\de t$
and complete the square for the right hand side of Eq.~\eqref{eq:ContCost}, thus getting the integral
$\int(A_r^* x(t)-\dot{x}_r(t))^2\de t -\int \dot{x}_r(t)^2\, \de t$. 
The first term is a sum of square residuals, and the second is independent of 
$A$. Finally the $\ell_1$ regularization term in Eq.~(\ref{eq:cont_reg_prob})
has the role of shrinking to $0$ a subset of the entries $A_{ij}$ thus
effectively selecting the structure.

Let $S^0$ be the support of row $A^0_r$, and assume $|S^0| \le k$.
We will refer to the vector $\sign(A^0_r)$ as to the \emph{signed support}
of $A^0_r$ (where $\sign(0)=0$ by convention). Let $\lambda_{\rm max}(M)$ and $\lambda_{\rm min}(M)$ stand for the maximum and minimum eigenvalue of a square matrix $M$ respectively. Further, denote by $A_{\min}$ the smallest absolute value among the non-zero entries of row $A^0_r$. 

When stable, the diffusion process (\ref{eq:BasicModel}) has a unique stationary 
measure which is Gaussian with covariance $Q^0\in\reals^{p\times p}$
given by the solution of Lyapunov's equation  \cite{Control}
\begin{equation}
A^0 Q^0 +  Q^0 (A^0)^* + I = 0 \label{eq:Lyapunov}.
\end{equation}
Our guarantee for regularized 
least squares is stated in terms of two properties
of the covariance $Q^0$ and one assumption on $\rho_{\rm min}(A^0)$
(given a matrix $M$, we denote by $M_{L,R}$ its submatrix $M_{L,R}
\equiv (M_{ij})_{i\in L, j\in R}$):
\begin{enumerate}
\item[$(a)$] We denote by $C_{\min} \equiv \lambda_{\min} (Q^0_{S^0,S^0})$
the minimum eigenvalue of the restriction of $Q^0$ to the support $S^0$ and assume $C_{\min} >0 $.
\item[$(b)$] We define the incoherence parameter $\alpha$ by letting
$\III  {Q^0}_{(S^0)^C, {S^0}} \left({Q^0}_{{S^0},{S^0}}\right)^{-1}  
\III _\infty = 1 - \alpha$, and assume $\alpha>0$. 
(Here $\III\, \cdot\,\III_{\infty}$ is the operator sup norm.)
\item[$(c)$] We define $\rho_{\rm min}(A^0) = -\lambda_{\rm max} ((A^0 + {A^0}^*)/2)$  and assume $\rho_{\rm min}(A^0)>0$. Note this is a stronger
form of stability assumption.
\end{enumerate}
Our main result is to show that there exists a well defined \emph{time complexity},
i.e. a minimum time interval $T$ such that, observing the system for time $T$
enables us to reconstruct the network with high probability. This result is stated in the following theorem.

\begin{theorem}\label{th:main_cont_time}
Consider the problem of learning the support $S^0$ of row $A^0_r$
of the matrix $A^0$ from a sample trajectory $\{x(t)\}_{t \in [0,T]}$ distributed
according to the model (\ref{eq:BasicModel}).  
If 
\begin{equation} 
T > \frac{10^4  k^2 (k\, \rho_{\rm min} (A^0)^{-2} + A_{\min}^{-2})}{\alpha^{2}
\rho_{\rm min} (A^0) C_{\min}^{2} }  \, \log\Big(\frac{4p k}{\delta}\Big)\, ,
 \label{eq:sample_bound_cont}
\end{equation}
then there exists $\lambda$ such that 
$\ell_1$-regularized least squares recovers the 
signed support of $A^0_r$  with probability larger than $1-\delta$.
This is achieved by taking
%
%
$\lambda = 
 \sqrt{{36\,\log(4p/\delta)}/({T\alpha^2 \rho_{\min} (A^0)})}\,\, .$
%
%
\end{theorem}
The time complexity is logarithmic in the number of variables and 
polynomial in the support size. Further, it is roughly inversely
proportional to $\rho_{\rm min}(A^0)$, which is quite satisfying
conceptually, since  $\rho_{\rm min}(A^0)^{-1}$
controls the relaxation time of the mixes.
%
%
\subsection{Overview of other results}

So far we focused on continuous-time dynamics. While, this is useful in order
to obtain elegant statements, much of the paper is in fact devoted to the
analysis of the following discrete-time dynamics, with parameter $\eta>0$:
\begin{eqnarray}
x(t) = x(t-1)+\eta A^0 x(t-1) + \, w(t), \;\;\;\;\; t \in \naturals_0\, .
\label{eq:DiscreteTimeModel}
\end{eqnarray}
Here $x(t)\in \reals^p$ is the vector collecting the dynamical
variables, $A^0\in\reals^{p\times p}$ specifies the dynamics as 
above, and $\{w(t)\}_{t\ge 0}$ is a sequence of i.i.d.
normal vectors with covariance $\eta\, I_{p\times p}$ (i.e.
with independent components of variance $\eta$).
We assume that consecutive samples $\{x(t)\}_{0\le t\le n}$ are
given and will ask under which conditions 
regularized least squares reconstructs the support of $A^0$. 

The parameter $\eta$ has the meaning of a time-step size. The 
continuous-time model (\ref{eq:BasicModel}) is recovered,
in a sense made precise below, by letting $\eta\to 0$. Indeed 
we will prove reconstruction guarantees that are uniform in this
limit as long as the product $n\eta$ (which corresponds to
the time interval $T$ in the previous section) is kept constant.
For a formal statement we refer to Theorem \ref{th:main_discrete}.
Theorem \ref{th:main_cont_time} is indeed proved by carefully 
controlling this limit. The mathematical challenge in this problem
is related to the fundamental fact that the samples $\{x(t)\}_{0\le t\le n}$
are dependent (and strongly dependent as $\eta\to 0$).

Discrete time models of the form (\ref{eq:DiscreteTimeModel}) can arise 
either because the system under study evolves by discrete steps,
or because we are subsampling a continuous time system
modeled as in Eq.~(\ref{eq:BasicModel}). Notice that in the latter case 
the matrices $A^0$  appearing in Eq.~(\ref{eq:DiscreteTimeModel})
and (\ref{eq:BasicModel}) coincide only to the zeroth order in $\eta$.
Neglecting this technical complication, the uniformity of 
our reconstruction guarantees as $\eta\to 0$ has an appealing
interpretation already mentioned above. Whenever the samples spacing
is not too large, the time complexity (i.e. the product $n\eta$) 
is roughly independent of the spacing itself.
%
%
\subsection{Related work}
\label{sec:Related}

A substantial amount of work has been devoted to the analysis of
$\ell_1$ regularized least squares, and its variants \cite{tibshirani_lasso, donoho2006nearsol, donoho2006sol, zhang2009ssp, wainwright2009sth}. The most 
closely related results are the one concerning high-dimensional 
consistency for support recovery \cite{wainwright2007high,zhao}. 
Our proof follows indeed the line of work developed in these papers,
with two important challenges. First, the design matrix is in our case
produced by a stochastic diffusion, and it does not necessarily
satisfies the irrepresentability conditions used by these works. Second,
the observations are not corrupted by i.i.d. noise (since successive 
configurations are correlated) and therefore elementary concentration 
inequalities are not sufficient.

Learning sparse graphical models via $\ell_1$ regularization is also 
a topic with significant literature. In the Gaussian case, 
the \emph{graphical} LASSO was proposed to reconstruct the model from i.i.d.
samples \cite{friedman2008sparse}. In the context of binary pairwise graphical models, 
Ref. \cite{wainwright2007high} proves high-dimensional consistency of
regularized logistic regression for structural learning,
under a suitable irrepresentability conditions on a modified covariance.
Also this paper focuses on i.i.d. samples. 

Most of these proofs builds on the technique of \cite{zhao}.
A naive adaptation to the present case 
allows to prove some performance guarantee  
for the discrete-time setting. However the resulting 
bounds are not uniform as $\eta\to 0$ for $n\eta = T$ fixed.
In particular, they do not allow to prove an analogous of
our continuous time result, Theorem \ref{th:main_cont_time}.
A large part of our effort is devoted to producing more accurate 
probability estimates that capture the correct scaling for small
$\eta$.

Similar issues were explored in the study of stochastic
differential equations, whereby one is often interested 
in tracking some slow degrees of freedom while `averaging out'
the fast ones \cite{Kurtz}. The relevance of this time-scale separation
for learning was addressed in \cite{Stuart}. Let us however emphasize
that these works focus once more on system with a fixed (small)
number of dimensions $p$.

Finally, the related topic of learning graphical models for
autoregressive processes was studied recently in 
\cite{Songsiri1,Songsiri2}. The convex relaxation proposed in these papers
is different from the one developed here. Further, no 
model selection guarantee was proved in \cite{Songsiri1,Songsiri2}.
%
%
\section{Illustration of the main results}

It might be difficult to get a clear intuition
of Theorem \ref{th:main_cont_time}, mainly because of conditions 
$(a)$ and $(b)$, which introduce parameters $C_{\min}$ and $\alpha$.
The same difficulty arises with analogous results on the high-dimensional
consistency of the LASSO \cite{wainwright2007high, zhao}.
In this section we provide concrete illustration both via numerical 
simulations, and by checking the condition on specific classes of graphs.
%
%
\subsection{Learning the laplacian of graphs with bounded degree}

Given a simple graph $\cG=(\cV,\cE)$ on 
vertex set $\cV=[p]$, its laplacian $\Delta^{\cG}$  is
the symmetric $p\times p$ matrix which is equal to the adjacency matrix
of $\cG$ outside the diagonal, and with entries $\Delta^{\cG}_{ii}=
-\dg(i)$ on the diagonal \cite{Chung}. 
(Here $\dg(i)$ denotes the degree of vertex $i$.)

It is well known that $\Delta^{\cG}$ is negative semidefinite, 
with one eigenvalue equal to $0$, whose multiplicity is equal to the
number of connected components of $\cG$.
The matrix $A^0 = -m\, I + \Delta^{\cG}$ fits into the setting of
Theorem \ref{th:main_cont_time} for $m>0$. 
The corresponding model (\ref{th:main_cont_time}) describes the 
over-damped dynamics of a network of masses connected by springs
of unit strength, and connected by a spring of strength $m$ to the origin.
We obtain the following result.
\begin{theorem}\label{th:cont_reg_graph_bounded_degree}
Let $\cG$ be a simple connected graph of maximum vertex degree $k$ and consider the model
(\ref{th:main_cont_time}) with 
$A^0 = -m\, I + \Delta^{\cG}$ where $\Delta^{\cG}$ is the laplacian of $\cG$
and $m>0$.
If
\begin{equation}
T \ge 2\cdot  10^5 k^2\, \Big(\frac{k+m}{m}\Big)^{5} (k + m^2)\, 
\log \Big(\frac{4pk}{\delta}\Big)\, ,
\end{equation}
then there exists $\lambda$ such that 
$\ell_1$-regularized least squares recovers the 
signed support of $A^0_r$  with probability larger than $1-\delta$.
 This is achieved by taking
$\lambda = \sqrt{36(k+m)^2\log(4p/\delta)/(T m^3)}$.
\end{theorem}
In other words, for $m$ bounded away
from $0$ and $\infty$, regularized least squares regression correctly
reconstructs the graph $\cG$ from a trajectory of time length 
which is polynomial in the degree and logarithmic in the system size.
Notice that once the graph is known, the laplacian $\Delta^{\cG}$ is 
uniquely determined. Also, the proof technique used for this example 
is generalizable to other graphs as well.
%
%
\subsection{Numerical illustrations}

In this section we present numerical validation of the proposed method on
synthetic data. 
 The results confirm our observations in Theorems 
\ref{th:main_cont_time} and \ref{th:main_discrete}, below,
namely that the time complexity scales logarithmically with the number 
of nodes in the network $p$, given a constant maximum degree.
Also, the time complexity is roughly independent of the sampling rate.
\begin{figure}[t]
\centering
\subfloat{
\includegraphics[width = .45 \textwidth]{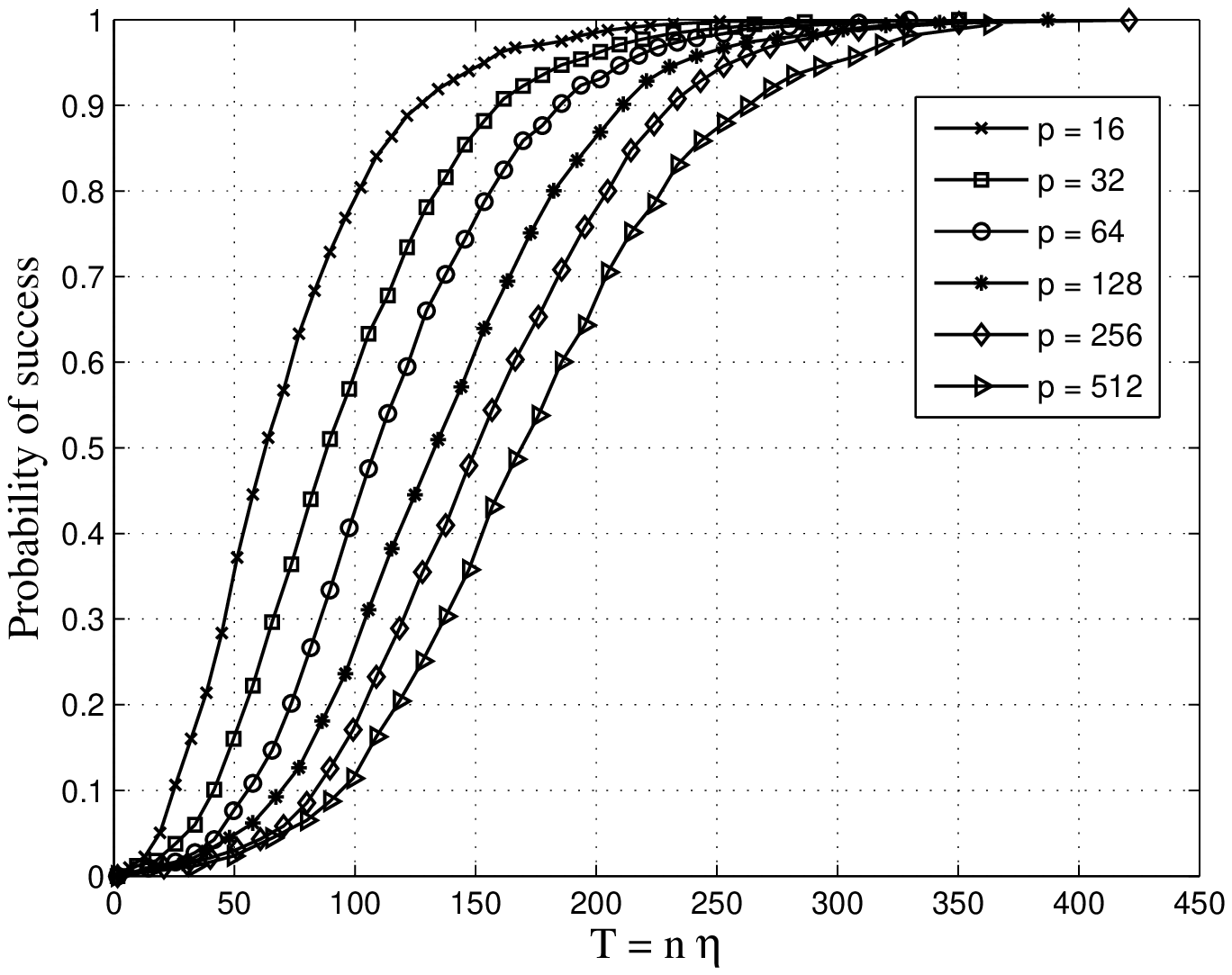}}
\subfloat{
\includegraphics[width =  .45 \textwidth]{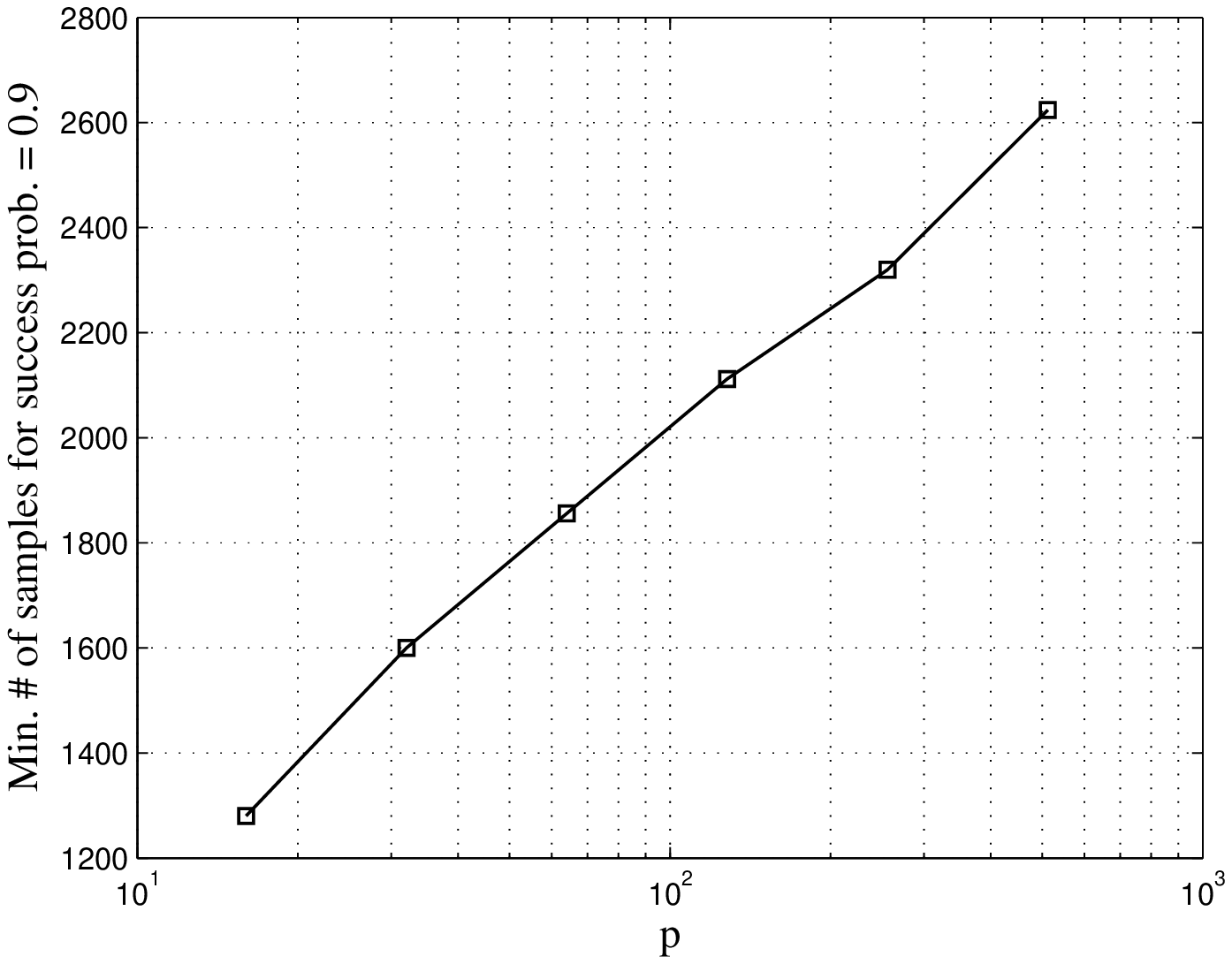}}
\caption{(left) Probability of success vs. length of the observation interval $n\eta$. (right) Sample complexity for 90\% probability of success vs. p.}
\label{fig:psVsN}
\end{figure}
In Fig.~\ref{fig:psVsN} and \ref{fig:psVsNeta_eta2} we consider the discrete-time
setting, generating data as follows. We draw  $A^0$ as a 
random sparse matrix in $\{0,1\}^{p \times p}$ with elements chosen 
independently at random with $\prob(A^0_{ij}=1) = k/p$, $k=5$. The process 
$x_0^n\equiv\{x(t)\}_{0\le t\le n}$ is then generated according to 
Eq.~\eqref{eq:DiscreteTimeModel}. 
We solve the regularized least square problem 
(the cost function is given explicitly in Eq.~\eqref{eq:discr_reg_prob}
for the discrete-time case) for different values of $n$, the number 
of observations, and record if the correct support is recovered for a random 
row $r$ using the optimum value of the parameter $\lambda$.  An estimate of
 the probability of successful recovery is obtained by repeating this 
experiment. Note that we are estimating here an average probability of success
over randomly generated matrices. 

The left plot in Fig.\ref{fig:psVsN} depicts the probability of success vs. $n\eta$ 
for $\eta =0.1$ and different values of $p$. Each curve is obtained using
$2^{11}$ instances, and each instance is generated using a new random 
matrix $A^0$. The right plot in Fig.\ref{fig:psVsN} is the corresponding curve of the sample complexity vs. $p$ where sample complexity is defined as the minimum value of $n\eta$ with probability of success of 90\%. As predicted by Theorem \ref{th:cont_reg_graph_bounded_degree} the curve shows the logarithmic scaling of the sample complexity with $p$. 

In Fig.~\ref{fig:psVsNeta_eta2} we turn to the continuous-time model 
(\ref{eq:BasicModel}). Trajectories are generated by discretizing
this stochastic differential equation with step $\delta$ much smaller
than the sampling rate $\eta$. 
We draw random matrices $A^0$ as above and plot the probability of success 
for $p=16$, $k=4$ and different values of $\eta$, as a function of $T$. We used $2^{11}$ instances for each curve. As predicted by Theorem \ref{th:main_cont_time},
for a fixed observation interval $T$, the probability of success 
converges to some limiting value as $\eta\to 0$.
%
%
\section{Discrete-time model: Statement of the results}
\label{sec:DiscreteResults}

Consider a system evolving in discrete time according to 
the model (\ref{eq:DiscreteTimeModel}),
and let $x_0^n\equiv\{x(t)\}_{0\le t\le n}$ be the observed portion of the trajectory.
The $r^{\text{th}}$ row $A^0_r$ is estimated by solving the following 
convex optimization problem for $A_r\in\reals^p$
\begin{equation}
{\rm minimize}\;\;\;
 L(A_r;x^n_{0}) + \lambda \| A_r\|_1\, ,\label{eq:discr_reg_prob}
\end{equation}
where 
\begin{equation}
L(A_r;x^n_0) \equiv \frac{1}{2\eta^2 n}\, \sum_{t=0}^{n-1}
\left\{x_r(t+1)-x_r(t)-\eta\, A_r^*x(t)\right\}^2\, .
\label{eq:opt_prob}
\end{equation}
\begin{figure}[t]
\centering
\subfloat{
\includegraphics[width = .45 \textwidth]{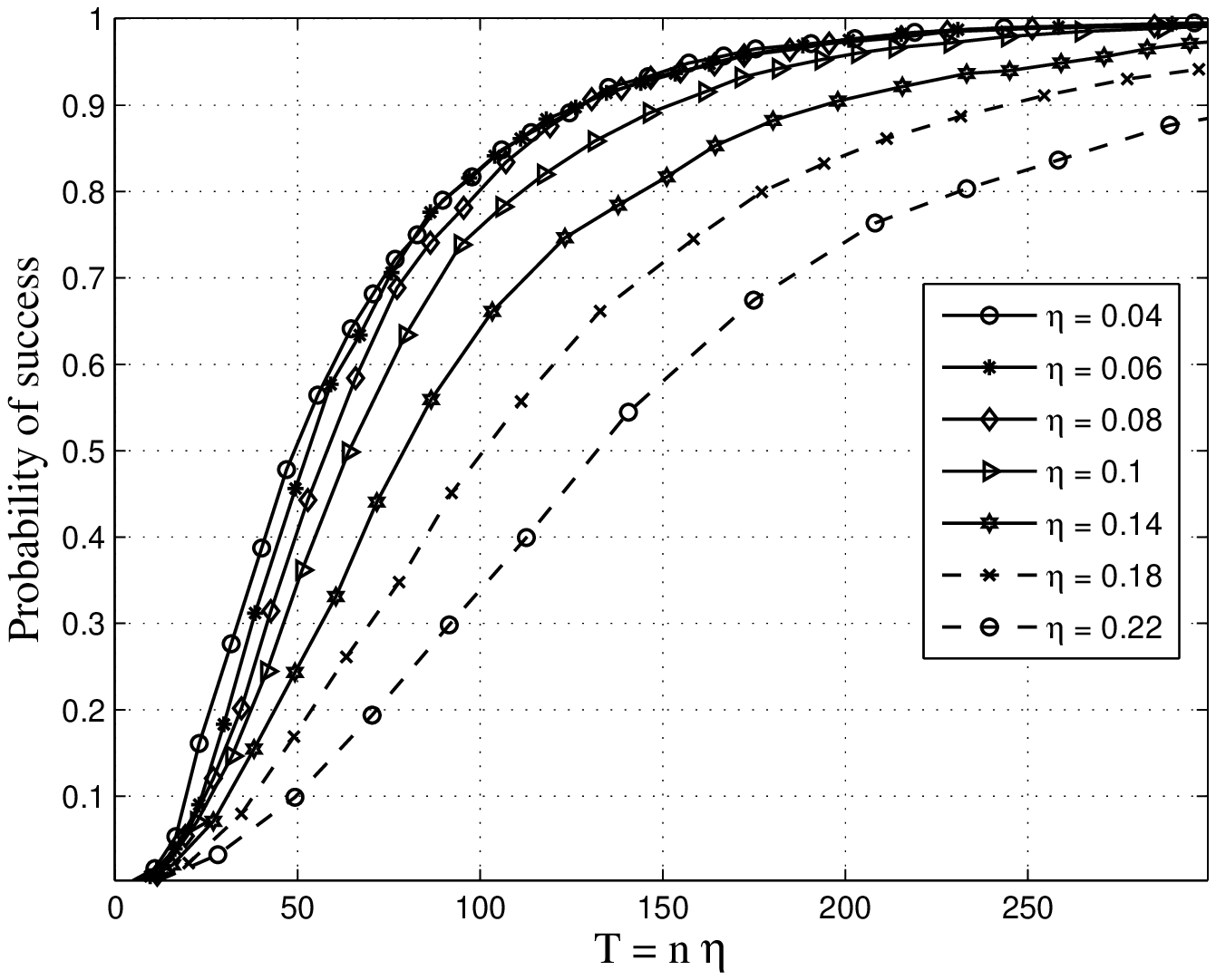}}
\qquad 
\subfloat{
\includegraphics[width =  .45 \textwidth]{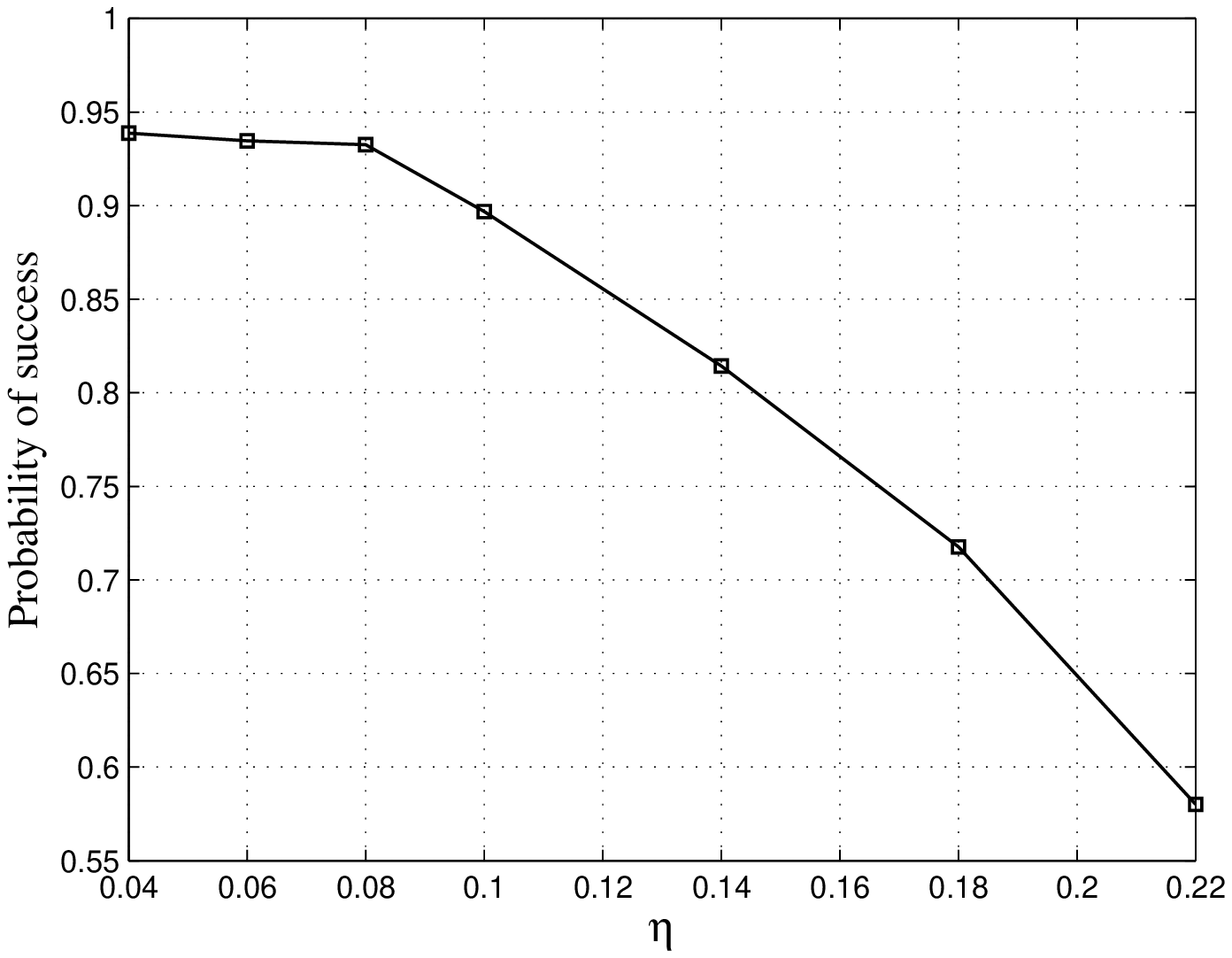}}
\caption{(right)Probability of success vs. length of the observation interval $n\eta$ for different values of $\eta$. (left) Probability of success vs. $\eta$ for a fixed length of the observation interval, ($n\eta = 150$) . The process is generated for a small value of $\eta$ and sampled at different rates.}
\label{fig:psVsNeta_eta2}
\end{figure}
Apart from an additive constant, the $\eta\to 0$ limit of this cost 
function can be shown to coincide with the cost function 
in the continuous time case, cf. Eq.~(\ref{eq:ContCost}). Indeed the proof
of Theorem \ref{th:main_cont_time} will amount to a more precise version
of this statement.
Furthermore, $L(A_r;x^n_0)$ is easily seen to be the log-likelihood of
$A_r$ within model (\ref{eq:DiscreteTimeModel}).

As before, we let $S^0$ be the support of row $A^0_r$, and assume
$|S^0| \le k$. Under the model (\ref{eq:DiscreteTimeModel}) $x(t)$ has a Gaussian stationary state distribution with covariance $Q^0$ determined by the following modified Lyapunov equation
\begin{equation}
A^0 Q^0 +  Q^0 (A^0)^* + \eta A^0Q^0 (A^0)^* +I = 0 \, .
\label{eq:Lyapunov2}
\end{equation}
It will be clear from the context whether $A^0$/$Q^0$ refers to the dynamics/stationary matrix from the continuous or discrete time system. We assume conditions $(a)$ and $(b)$ introduced in 
Section \ref{sec:ResContinuumRegularized}, and adopt the notations 
already introduced there. We use as a shorthand notation $\sigma_{\max} \equiv \sigma_{\max}(I+\eta\, A^0)$ where $\sigma_{\max}(.)$ is the maximum singular value. Also define
%
%
$D \equiv \big(1-\sigma_{\max}\big)/\eta\, .$
%
%
We will assume $D>0$. As in the previous section, we assume the model \eqref{eq:DiscreteTimeModel} is initiated in the stationary state.
\begin{theorem}\label{th:main_discrete}
Consider the problem of learning the support $S^0$ of row $A^0_r$
from the discrete-time trajectory $\{x(t)\}_{0\le t\le n}$. If
\begin{equation} 
n \eta > \frac{10^4  k^2 (k D^{-2} + A_{\min}^{-2})}{\alpha^{2} D C_{\min}^{2}}
\, \log\Big(\frac{4p k}{\delta}\Big)\, ,\label{eq:sample_bound_disc}
\end{equation}
then  there exists $\lambda$ such that 
$\ell_1$-regularized least squares recovers the 
signed support of $A^0_r$  with probability larger than $1-\delta$.
This is achieved by taking
$\lambda = \sqrt{(36\,\, \log(4p/\delta))/(D\alpha^2 n\eta)  }$.
\end{theorem}
In other words the discrete-time sample complexity, $n$, is logarithmic 
in the model dimension, polynomial in the maximum network degree
and inversely proportional to the time spacing between samples.
The last point is particularly important. It enables us to derive the bound on the continuous-time sample complexity as the limit $\eta \rightarrow 0$ of the discrete-time sample complexity. It also confirms our intuition mentioned in the Introduction: although one can produce an arbitrary large number of samples by sampling the continuous process with finer resolutions, there is limited amount of information that can be harnessed from a given time interval $[0,T]$.
%
%
\section{Proofs}
%

In the following we denote by $X\in\reals^{n\times p}$
the matrix whose $(t+1)^{\text{th}}$ column corresponds to the configuration $x(t)$,
i.e. $X = [x(0), x(1), \dots, x(n-1)]$. 
Further  $\Delta X \in \reals^{n\times p}$ 
is the matrix containing configuration changes, namely  
$\Delta X=[x(1)-x(0), \dots, x(n)-x(n-1)]$.
Finally we write $W = [w(1),\dots,w(n-1)]$ for the matrix containing 
the Gaussian noise realization. Equivalently, 
\begin{eqnarray*}
W = \Delta X-\eta A\, X\, .
\end{eqnarray*}
The $r^{\rm th}$ row of $W$ is denoted by $W_r$.

In order to lighten the notation, we will omit the reference to 
$x_0^n$ in the likelihood function (\ref{eq:opt_prob}) 
and simply write $L(A_r)$. 
We define its normalized gradient and Hessian by
\begin{eqnarray}
\hG = - 
\nabla L (A^0_r) =  \frac{1}{n \eta} X W_r^*\, ,\;\;\;\;\;\;\;
\hQ =  \nabla^2 L (A^0_r) =  \frac{1}{n} X X^*\, .
\end{eqnarray}
%
%
\subsection{Discrete time}

In this Section we outline our prove for our main result for discrete-time
dynamics, i.e., Theorem \ref{th:main_discrete}. We start by stating
a set of sufficient conditions for regularized least squares to work.
Then we present a series of concentration lemmas to be used to prove the validity of these conditions, and finally we sketch the outline of the proof. 

As mentioned, the proof strategy, and in particular the following proposition which provides a compact set of sufficient conditions for the support to be recovered correctly is analogous to the one in 
\cite{zhao}. A proof of this proposition can be found in the supplementary material. 
\begin{proposition}\label{th:cond_to_hold}
Let $\alpha, C_{\min}>0$ be  be defined by
\begin{eqnarray}
\lambda_{\min}(Q^0_{{S^0},{S^0}}) \equiv C_{\min} \,,
\qquad
  \III Q^0_{(S^0)^C, {S^0}} \left(Q^0_{{S^0},{S^0}}\right)^{-1} \III _\infty \equiv
 1 - \alpha\, . \label{eq:cond_irrep}
\end{eqnarray}
If the following conditions hold then the
regularized least square solution \eqref{eq:discr_reg_prob} 
correctly recover the 
signed support $\sign(A^0_r)$:
\begin{align}
\|\hG\|_\infty &\leq \frac{\lambda \alpha}{3}\, ,\;\;\;\;\;\;\;\;\;\;\;\;\;
\|\hG_{S^0}\|_\infty \leq \frac{A_{\min} C_{\min}}{4 k} - \lambda,\label{eq:Conditions1}\\
 \III \hQ_{(S^0)^C ,{S^0}} - Q^0_{(S^0)^C, {S^0}} \III _\infty &\leq
\frac{\alpha}{12} \frac{C_{\rm min}}{\sqrt{k}}\, 
,\;\;\;\;\;\;\;
 \III \hQ_{{S^0},{S^0}} - Q^0_{{S^0},{S^0}} \III _\infty \leq \frac{\alpha}{12} \frac{C_{\min}}{\sqrt{k}}\, .\label{eq:matrix_norm_cond_12}
\end{align}
Further the same statement holds for the continuous model \ref{eq:ContCost},
provided  $\hG$ and $\hQ$ are the gradient and the hessian of the 
likelihood (\ref{eq:ContCost}).
\end{proposition}

The proof of Theorem \ref{th:main_discrete}
consists in checking that, under the hypothesis 
\eqref{eq:sample_bound_disc} on the number of consecutive configurations,
conditions (\ref{eq:Conditions1}) to (\ref{eq:matrix_norm_cond_12})
will hold with high probability.
Checking these conditions can be regarded in turn as 
concentration-of-measure statements. Indeed, if expectation is taken with 
respect to a stationary trajectory, we have
$\E\{\hG\}=0$, $\E\{\hQ\}= Q^0$. 
%
%
\subsubsection{Technical lemmas}

In this section we will state the necessary concentration lemmas for proving Theorem \ref{th:main_discrete}. These are non-trivial because $\hG$, $\hQ$ are quadratic  
functions of \emph{dependent} random variables $\big($the samples $\{x(t)\}_{0 \le t\le n}$$\big)$. 
The proofs of Proposition \ref{th:gradie_prob_bound}, of Proposition \ref{en:bddQij}, and Corollary \ref{en:bddQJS} can be found in the supplementary material provided.

Our first Proposition implies concentration of $\hG$ around $0$.
\begin{proposition} \label{th:gradie_prob_bound}
Let $S\subseteq [p]$ be any set of vertices and $\epsilon<1/2$.
If $\sigma_{\rm max}\equiv\sigma_{\max}(I+\eta\, A^0)<1$, then 
\begin{eqnarray}
\prob\big\{\|\hG_S\|_\infty > \epsilon\big\} \leq 2 |S|\, e^{-n (1-\sigma_{\max})\, \epsilon^2/4}.
\end{eqnarray}
\end{proposition}
We furthermore need to bound the matrix norms as per \eqref{eq:matrix_norm_cond_12}  in proposition $\ref{th:cond_to_hold}$. First we relate bounds on $ \III \hQ_{JS} - {Q^0}_{JS}  \III_\infty$ with bounds on $|\hQ_{ij} - Q^0_{ij}|$, ($i \in J, i \in S$) where $J$ and $S$ are any subsets of $\{1,...,p\}$.  We have,
\begin{equation}\label{eq:UnionBound}
\prob ( \III  \hQ_{JS} - Q^0_{JS})  \III _{\infty} > \epsilon )
\leq |J| |S| \max_{i,j \in J} \prob ( | \hQ_{ij} - Q^0_{ij} | > \epsilon / |S|).
\end{equation}
Then, we bound $|\hQ_{ij} - Q^0_{ij}|$ using the following proposition
\begin{proposition}
Let $i,j\in \{ 1,...,p\}$,  $\sigma_{\max} \equiv \sigma_{max}(I + \eta A^0) < 1$, $T = \eta n > 3/D$ and $0 < \epsilon< 2/D$ where $D = (1-\sigma_{\max})/\eta$ then, 
\begin{eqnarray}
\prob(|\hQ_{ij} - Q^0_{ij})|  > \epsilon) \leq   2 e^{-\frac{n}{32 \eta^2}  (1-\sigma_{\max})^3 \epsilon^2}.
\end{eqnarray} \label{en:bddQij}
\end{proposition}
Finally,  the next corollary follows
from Proposition \ref{en:bddQij} and Eq.~\eqref{eq:UnionBound}.
\begin{corollary}\label{en:bddQJS}
Let $J,S$ ($|S| \leq k$) be any two subsets of $\{ 1,...,p\}$ and  $\sigma_{\max} \equiv \sigma_{\max}(I + \eta A^0)<1$, $\epsilon< 2 k/D$ and $n\eta > 3/D$ (where $D = (1 - \sigma_{\max})/\eta )$ then, 
\begin{eqnarray}
\prob( \III \hQ_{JS} - Q^0_{JS} \III _{\infty}  > \epsilon) \leq   2 |J| k e^{-\frac{n}{32 k^2 \eta^2} (1-\sigma_{\max})^3\epsilon^2}.
\end{eqnarray} \label{th:prob_bound_inft_matrix_hess}
\end{corollary}
%
%
\subsubsection{Outline of the proof of Theorem \ref{th:main_discrete}}

With these concentration bounds we can now easily prove Theorem \ref{th:main_discrete}. All we need to do is to compute the probability that the conditions given by Proposition \ref{th:cond_to_hold} hold.
From the statement of the theorem we have that the first two conditions ($\alpha, C_{\min} > 0$) of Proposition \ref{th:cond_to_hold} hold. In order to make the first condition on $\hG$ imply the second condition on $\hG$ we assume that 
$\lambda \alpha/3 \leq (A_{\min} C_{\min})/(4 k) - \lambda$
which is guaranteed to hold if
\begin{equation}
\lambda \leq A_{\min} C_{\min} / 8k. \label{eq:th_main_disc_lambda} 
\end{equation}
We also combine the two last conditions on $\hQ$, thus obtaining the following
\begin{equation}
 \III \hQ_{[p], {S^0}} - Q^0_{[p], {S^0}} \III _\infty \leq \frac{\alpha}{12} \frac{C_{\min}}{\sqrt{k}}\, , 
\end{equation}
since $[p]=S^0 \cup (S^0)^C$. We then impose that both the probability of the condition on $\hQ$ failing and the probability of the condition on $\hG$ failing are upper bounded by $\delta/2$ using Proposition \ref{th:gradie_prob_bound} and Corollary \ref{th:prob_bound_inft_matrix_hess}. 
It is shown in the supplementary material that this is satisfied if
condition (\ref{eq:sample_bound_disc}) holds.

%
%
\subsection{Outline of the proof of Theorem \ref{th:main_cont_time}}

\comment{
\begin{proofof}{Theorem \ref{th:main_cont_time}}}
To prove Theorem \ref{th:main_cont_time} we recall that 
Proposition \ref{th:cond_to_hold} holds provided the appropriate
continuous time expressions are used for $\hG$ and $\hQ$, namely
\begin{eqnarray}
\hG = - 
\nabla \cL (A^0_r) =  \frac{1}{T}\int_0^T\! x(t)\, \de b_r(t)\, ,\;\;\;\;\;\;\;
\hQ =  \nabla^2 \cL (A^0_r) =   \frac{1}{T}\int_0^T\! x(t)x(t)^*
\, \de t \, .
\end{eqnarray}
These are of course random variables. In order to distinguish these from 
the discrete time version, we will adopt the notation $\hG^n$, $\hQ^n$
for the latter. We claim that these random variables can be coupled
(i.e. defined on the same probability space) in such a way that 
$\hG^n\to \hG$ and $\hQ^n\to\hQ$ almost surely as $n\to\infty$
for fixed $T$. Under assumption (\ref{eq:sample_bound_cont}),
it is easy to show that (\ref{eq:sample_bound_disc}) holds for all $n>n_0$
with $n_0$ a sufficiently large constant (for a proof see the provided supplementary material).
Therefore, by the proof of Theorem \ref{th:main_discrete},
the conditions in Proposition \ref{th:cond_to_hold} hold 
for gradient $\hG^n$ and hessian $\hQ^n$ for any $n\ge n_0$, with probability 
larger than $1-\delta$. But by the claimed convergence 
$\hG^n\to \hG$ and $\hQ^n\to\hQ$, they hold also for $\hG$ and $\hQ$
with probability at least $1-\delta$ which proves the theorem.

We are left with the task of showing that the discrete and continuous
time processes can be coupled in such a way that $\hG^n\to \hG$ 
and $\hQ^n\to\hQ$.
With slight abuse of notation, the state of the discrete time system 
\eqref{eq:DiscreteTimeModel} will be denoted by $x(i)$ where 
$i \in \naturals$ and the state of continuous time system 
\eqref{eq:BasicModel} by $x(t)$ where $t \in \reals$. We denote by $Q^0$ the solution of \eqref{eq:Lyapunov} and by $Q^0(\eta)$ the solution of \eqref{eq:Lyapunov2}. It is easy to check that $Q^0(\eta) \rightarrow Q^0$ as $\eta \rightarrow 0$ by the uniqueness of stationary state distribution.

The initial state of the continuous time system $x(t=0)$ is a 
$\normal(0, Q^0)$ random variable independent of $b(t)$ and the initial 
state of the discrete time system is defined to be 
$x(i=0) = (Q^0(\eta))^{1/2} (Q^0)^{-1/2} x(t = 0)$.
At subsequent times, $x(i)$ and $x(t)$ are assumed are generated by 
the respective dynamical systems using the same matrix $A^0$
using common randomness provided by the standard  Brownian motion 
$\{b(t)\}_{0\le t\le T}$ in $\reals^p$.
In order to couple $x(t)$ and 
$x(i)$, we construct $w(i)$, the noise 
driving the discrete time system, by letting 
$w(i)\equiv (b(Ti/n) - b(T(i-1)/n))$. 

The almost sure convergence $\hG^n\to \hG$ and $\hQ^n\to\hQ$
follows then from standard convergence of random walk to Brownian motion.

\section*{Acknowledgments}
This work was partially supported by a Terman fellowship, the NSF CAREER award CCF-0743978 and the NSF grant DMS-0806211 and by a Portuguese Doctoral FCT fellowship.
\newpage
\bibliographystyle{unsrt}
\bibliography{inference}

\newpage

\appendix

\section{Learning networks of  stochastic differential equations: Supplementary materials}

In order to prove Proposition \ref{th:cond_to_hold} we first introduce two technical lemmas.

\begin{lemma}
For any subset $S \subseteq [p]$ the following decomposition holds,
\begin{equation}
\hQ_{S^C, S} \left(\hQ_{S,S}\right)^{-1} = T_1 + T_2 + T_3 + Q^0_{S^C, S} \left(Q^0_{S,S}\right)^{-1},
\end{equation}
where,
\begin{eqnarray}
T_1 &=& Q^0_{S^C, S} \left( \left(\hQ_{S,S}\right)^{-1} - \left(Q^0_{S,S}\right)^{-1}\right), \\
T_2 &=& (\hQ_{S^C, S} - Q^0_{S^C, S} ) \left(Q^0_{S,S}\right)^{-1},\\
T_3 &=& (\hQ_{S^C, S} - Q^0_{S^C, S})\left (\left(\hQ_{S,S}\right)^{-1} - \left(Q^0_{S,S}\right)^{-1}\right).\\
\end{eqnarray}
In addition, if $ \III Q^0_{S^C, S} \left(Q^0_{S,S}\right)^{-1} \III _\infty < 1$ and $\lambda_{\min}(\hQ_{S,S}) \geq C_{\min}/2>0$ the following relations hold,
\begin{eqnarray}
 \III T_1 \III _\infty &\leq& \frac{2 \sqrt{k} }{C_{\min}} \III \hQ_{S, S} - Q^0_{S, S} \III _\infty,\\
 \III T_2 \III _\infty &\leq& \frac{\sqrt{k} }{C_{\min}} \III \hQ_{S^C, S} - Q^0_{S^C, S} \III _\infty,\\
 \III T_3 \III _\infty &\leq& \frac{2 \sqrt{k} }{C_{\min}^2} \III \hQ_{S^C, S} - Q^0_{S^C, S} \III _\infty \III \hQ_{S, S} - Q^0_{S, S} \III _\infty.
\end{eqnarray} \label{th:inco_decomp}
\end{lemma}

The following lemma taken from the proofs of Proposition 1 in \cite{ravikumar2008high} and Proposition 1 in \cite{zhao} respectively is the crux to guaranteeing correct signed-support reconstruction of $A_r^0$.

\begin{lemma}
If $\hQ_{{S^0},{S^0}} > 0$, then the dual vector $\hat{z}$ from the KKT conditions of the optimization problem \eqref{eq:discr_reg_prob} satisfies the following inequality,
\begin{equation}
\|\hat{z}_{(S^0)^C}\|_\infty  \leq  \III  \hQ_{(S^0)^C, {S^0}} \left(\hQ_{{S^0}, {S^0}}\right)^{-1}  \III _\infty \left( 1 +  \frac{\|\hG_{S^0}\|_\infty}{\lambda}  \right) + \frac{\|\hG_{(S^0)^C}\|_\infty}{\lambda}.
\end{equation}
In addition, if 
\begin{equation}
\|\hG_{S^0}\|_\infty \leq \frac{A_{\min} \lambda_{\min}(\hQ_{{S^0},{S^0}}) }{2 k} - \lambda
\end{equation}
then $\|A^0_r - \hat{A}_r\|_\infty \leq A_{\min}/2$. The same result holds for problem \eqref{eq:cont_reg_prob}.
\label{th:main_cond}
\end{lemma}

\begin{proofof}{Proposition \ref{th:cond_to_hold}}\label{pf:compact_condition}
To guarantee that our estimated support is at least contained in the true support we need to impose that $\|\hat{z}_{S^C}\|_\infty < 1$. To guarantee that we do not introduce extra elements in estimating the support and also to determine the correct sign of the solution we need to impose that $\|A^0_r - \hat{A}_r\|_\infty \leq A_{\min}/2$.
Now notice that since $\lambda_{\min}(Q^0_{{S^0},{S^0}}) = C_{\min}$ the relation $\lambda_{\min}(\hQ_{{S^0},{S^0}}) \geq C_{\min}/2$ is guaranteed as long as $ \III \hQ_{{S^0}, {S^0}} - Q^0_{{S^0}, {S^0}} \III _\infty \leq C_{\min}/2$. Using Lemma \ref{th:inco_decomp} it is easy to see that the bounds of Proposition \ref{th:cond_to_hold} lead to the conditions of Lemma \ref{th:main_cond} being verified. Thus, these lead to a correct recovery of the signed structure of $A^0_r$.   
\end{proofof}

\begin{lemma}
Let $r,j \in [p]$ and let $\rho(\tau)$ represent a $p \times p$ matrix with all rows equal to zero except the $r^{th}$ row which equals the $j^{th}$ row of ${(I + \eta A^0)}^\tau$ (the $\tau^{th}$ power of $I + \eta A^0$ ). Let $\tilde{R}(j) \in \reals^{(n+m+1)\times(n+m+1)}$ be defined as,
\begin{eqnarray}
\tilde{R} =  \left(
\begin{array}{ccccccccc}
0&0& \ldots&0 &0 & 0 & \ldots &0 &0\\
\vdots&\vdots& \ddots&\vdots &\vdots & \vdots & \ddots & \vdots& \vdots\\
0&0& \ldots&0 &0 & 0 &\ldots & 0&0\\
\rho(m)&\rho(m-1)&\ldots&\rho(1)&\rho(0)&0&\ldots&0&0\\
\rho(m+1)&\rho(m)&\ldots&\rho(2)&\rho(1)&\rho(0)&\ldots&0&0\\
\vdots&\vdots&\ddots&\vdots&\vdots&\vdots&\ddots&0&0\\
\rho(m+n-1)&\rho(m+n-2)&\ldots&\rho(n)&\rho(n-1)&\rho(n-2)&\ldots&\rho(0)&0\\
\end{array}
\right).
\end{eqnarray}
Define $R(j) = 1/2 (\tilde{R} + \tilde{R}^*)$ and let $\nu_i$ denote its $i^{th}$ eigenvalue and assume $\sigma_{\max} \equiv \sigma_{\max} (I + \eta A^0) < 1$. Then,
\begin{align}
\sum^{p(n+m+1)}_{i=1}\nu_i &= 0,\label{eq:gradient_sum_eigen_bound}\\
\max_i |\nu_i|  &\leq \frac{1}{1-\sigma_{\max}}, \label{eq:gradient_eigen_bound}\\
\sum^{p(n+m+1)}_{i=1}\nu^2_i &\leq \frac{1}{2} \frac{n}{1-\sigma_{\max}}.\label{eq:gradient_sum_eigen2_bound}
\end{align} \label{th:matrix_proper_grad}
\end{lemma}
\begin{proof}
First it is immediate to see that $\sum^{p(n+m+1)}_{i=1}\nu_i = Tr(R) = 0$. Let ${I_1}_\tau$ represent a $p \times p$ matrix with zeros everywhere and ones in the block-position where $\rho(\tau)$ appears and ${I_2}_\tau$ represent a similar matrix but with ones in the block-position where $\rho(\tau)^*$ appears. Then $R$ can be written as,
\begin{equation}
R = \frac{1}{2}    \left(  \sum^{m+n-1}_{\tau =  0} {I_1}_\tau \otimes   \rho(\tau)   +  {I_2}_\tau \otimes   \rho(\tau)^*  \right),
\end{equation}
where $\otimes$ denotes the Kronecker product of matrices. 
This expression can be used to compute an upper bound on $|\nu_i|$. Namely,
\begin{align}
\max_i{|\nu_i|} &= \sigma_{\max}(R) \leq  \sum^\infty_{\tau = 0} \sigma_{\max} ( {I_1}_\tau \otimes   \rho(\tau) )\leq  \sum^\infty_{\tau = 0} \sigma_{\max} ( {I_1}_\tau)\sigma_{\max} (  \rho(\tau)  )\\
& \leq   \sum^\infty_{\tau = 0} \sigma_{\max} ( \rho(\tau) ) \leq  \sum^\infty_{\tau = 0} \sigma_{\max}^{\tau} =\frac{1}{1-\sigma_{\max} ( {\varphi^*})}.
\end{align}
For the other bound we do,
\begin{align}
\sum^{(n+m+1)p}_{i=1}\nu^2_i &= Tr(R^2) \leq \frac{1}{4} \; n \; 2 \, \sum^{\infty}_{\tau = 0} Tr(  \rho(\tau) \rho(\tau)^*  )\\
&=\frac{1}{2}  n  \sum^{\infty}_{\tau = 0} \|  \rho(\tau)  \|^2_2\\
& \leq \frac{1}{2} n  \sum^{\infty}_{\tau = 0} \sigma_{\max}^{2\tau} \leq \frac{1}{2} \frac{n}{1-\sigma_{\max}},
\end{align}
where in the last step we used the fact that $0 \leq \sigma_{\max} < 1$.
\end{proof}

\begin{lemma}
Let $j \in [p]$. Define $\rho(\tau) \in \reals^{1 \times p}$ to be the $j^{th}$ row of $(I + \eta A^0)^\tau$. Let $\Phi_j \in \reals^{n \times (n+m)}$ be defined as,
\begin{equation}
\Phi_j =    \left(
\begin{array}{ccccccccc}
\rho(m)&\rho(m-1)&\ldots&\rho(1)&\rho(0)&0&\ldots&0\\
\rho(m+1)&\rho(m)&\ldots&\rho(2)&\rho(1)&\rho(0)&\ldots&0\\
\vdots&\vdots&\ddots&\vdots&\vdots&\vdots&\ddots&0\\
\rho(m+n-1)&\rho(m+n-2)&\ldots&\rho(n)&\rho(n-1)&\rho(n-2)&\ldots&\rho(0)\\
\end{array}
\right),
\end{equation}
Let $\nu_l$ denote the $l^{th}$ eigenvalue of the matrix $R(i,j) = 1/2 (\Phi_j^* \Phi_i + \Phi_i^* \Phi_j) \in \reals^{(n+m) \times (n+m)} $ (where $i \in [p]$) and assume $\sigma_{\max} \equiv \sigma_{\max} (I + \eta A^0) < 1$ then,
\begin{align}
|\nu_l| &\leq \frac{1}{(1-\sigma_{\max})^2},\\
\frac{1}{n}\sum^{(n+m)p}_{l = 1} \nu^2_l &\leq \frac{2}{(1-\sigma_{\max})^3}\left(1 + \frac{3}{2n} \frac{1}{1-\sigma_{\max}} \right).
\end{align} \label{th:matrix_proper_hess}
\end{lemma}
\begin{proof}
The first bound can be proved in a trivial manner. In fact, since for any matrix $A$ and $B$ we have $\sigma_{\max}(A+B) \leq \sigma_{\max}(A) + \sigma_{\max}(B)$ and $\sigma_{\max}(AB) \leq \sigma_{\max}(A)\sigma_{\max}(B)$ we can write
\begin{align}
\max_l |\nu_l| &= \sigma_{\max} (1/2 (\Phi_j^* \Phi_i + \Phi_i^* \Phi_j)) \leq 1/2 (\sigma_{\max}(\Phi_j^* \Phi_i) +  \sigma_{\max}(\Phi_i^* \Phi_j))\\
&\leq \sigma_{\max}(\Phi_i^* \Phi_j) \leq \sigma_{\max}(\Phi_i) \sigma_{\max}(\Phi_j) \leq \frac{1}{(1-\sigma_{\max})^2},
\end{align}
where in the last inequality we used the fact $\sigma_{\max}(\Phi_j) \leq 1/ (1-\sigma_{\max})$. The proof of this is just a copy of the proof of the bound \eqref{eq:gradient_eigen_bound} in Lemma \ref{th:matrix_proper_grad}.

Before we prove the second bound let us introduce some notation to differentiate $\rho(\tau)$ associated with $\Phi_j$ from $\rho(\tau)$ associated with $\Phi_i$. Let us call them $\rho(\tau,j)$ and $\rho(\tau,i)$ respectively. Now notice that $\Phi_i^* \Phi_j$ can be written as a block matrix
\begin{equation}
\left(
\begin{array}{cccc}
\tilde{A} &  \tilde{D}\\
\tilde{C} &  \tilde{B}
\end{array}
\right)
\end{equation}
where $\tilde{A}, \tilde{B}, \tilde{C}$ and $\tilde{D}$ are matrix blocks where each block is a $p$ by $p$ matrix. $\tilde{A}$ has $p \times p$ blocks, $\tilde{B}$ has $n \times n$ blocks, $\tilde{C}$ has $n \times m$ blocks and $\tilde{D}$ has $m \times n$ blocks. If we index the blocks of each matrix with the indices $x,y$ these can be described in the following way
\begin{align}
\tilde{A}_{xy} &=\sum^m_{s=1} \rho(m-x+s,i)^* \rho(m-y+s,j) \\
\tilde{B}_{xy} &=\sum^{n-x}_{s=0} \rho(s,i)^* \rho(s+x-y,j),x \geq y\\
\tilde{B}_{xy} &=\sum^{n-y}_{s=0} \rho(s+y-x,i)^* \rho(s,j),x \leq y\\
\tilde{C}_{xy} &=\sum^{n-x}_{s=0} \rho(s,i)^* \rho(m-y+x+s,j)\\
\tilde{D}_{xy} &=\sum^{n-y}_{s=0} \rho(m-x+y+s,i)^* \rho(s,j).
\end{align}
With this in mind and denoting by $A,B,C$ and $D$ the symmetrized versions of these same matrices (e.g.: $A = 1/2 (\tilde{A} + \tilde{A}^*)$) we can write,
\begin{equation}
\sum^{(n+m)p}_{l=1} \nu^2_l = Tr(R(i,j)^2) = Tr(A^2) + Tr(B^2) + 2 Tr(CD). 
\end{equation}
We now compute a bound for each one of the terms. We exemplify in detail the calculation of the first bound only.
First write,
\begin{align}
Tr(A^2) &= \sum^m_{x=1} \sum^m_{y=1} Tr(A_{xy} A^*_{xy}).
\end{align}
Now notice that each $Tr(A_{xy} A^*_{xy})$ is a sum over $\tau_1,\tau_2  \in [p]$ of terms of the type,
\begin{eqnarray}
(\rho(m-x+\tau_1,i)^* \rho(m-y+\tau_1,j)+ \rho(m-x+\tau_1,j)^* \rho(m-y+\tau_1,i))\times\\
\times(\rho(m-y+\tau_2,j)^* \rho(m-x+\tau_2,i) + \rho(m-y+\tau_2,i)^* \rho(m-x+\tau_2,j)).
\end{eqnarray}
The trace of a matrix of this type can be easily upper bounded by
\begin{equation}
(\sigma_{\max})^{m - x + \tau_1 + m - y + \tau_1 + m - y + \tau_2 + m - x + \tau_2} = (\sigma_{\max})^{2(m-x) +2(m-y) +2 \tau_1 + 2\tau_2}
\end{equation}
which finally leads to
\begin{equation}
Tr(A^2) \leq \frac{1}{(1-\sigma_{\max})^4}.
\end{equation}
Doing a similar thing to the other terms leads to
\begin{align}
Tr(B^2) &\leq \sum^{n,n}_{x,y} \sum_{\tau_1,\tau_2} \sigma_{\max}^{2 \tau_1 + 2 \tau_2+ 2|x - y|} \leq \frac{2n}{(1 - \sigma_{\max})^3}\\
Tr(DC) &= \sum^m_{x=1}  \sum^n_{y=1} Tr(C_{xy} D_{yx})\leq \sum^{m,n,n-y,n-y}_{x,y,\tau_1,\tau_2} \sigma_{\max}^{2 (m-x) + 2y + 2\tau_1 + 2\tau_2} \leq \frac{1}{(1 - \sigma_{\max})^4}.
\end{align}
Putting all these together leads to the desired bound.
\end{proof}

\begin{proofof}{Proposition \ref{th:gradie_prob_bound}}
We will start by proving that this exact same bound holds when the probability of the event $\{ \|\hG_S\|_\infty > \epsilon\}$ is computed with respect to a trajectory $\{x(t)\}^n_{t=0}$ that is initiated at instant $t = -m$ with the value $w(-m)$. In other words, $x(-m) = w(-m)$. Assume we have done so. Now notice that as $m \rightarrow \infty$, $X$ converges in distribution to $n$ consecutive samples from the model \eqref{eq:DiscreteTimeModel} when this is initiated from stationary state. Since $\|\hG_S\|_\infty$ is a continuous function of $X = [x(0),...,x(n-1)]$, by the Continuous Mapping Theorem, $\|\hG_S\|_\infty$ converges in distribution to the corresponding random variable in the case when the trajectory $\{x(i)\}^n_{i=0}$ is initiated from stationary state. Since the probability bound does not depend on $m$ we have that this same bound holds for stationary trajectories too.

We now prove our claim. Recall that  $\hG_j = (X_j W_r^*)/(n\eta)$. Since $X$ is a linear function of
the independent gaussian random variables $W$ we can write
$X_j W_r^* = \eta  z^* R(j) z$, where $z\in \reals^{p (n+m+1)}$ is a 
vector of i.i.d. $\normal(0,1)$ random variables and $R(j) \in \reals^{p(n+m+1)\times
p(n+m+1)}$ is the symmetric matrix defined in Lemma \ref{th:matrix_proper_grad}.

Now apply the standard Bernstein method. First by union bound we have
\begin{align*}
\prob\big\{\|\hG_S\|_\infty > \epsilon\big\} &\leq  2|S| 
\,\max_{j\in S} \prob\big\{z^* R(j) z > n \epsilon\big\} \, .
\end{align*}
Next denoting by $\{\nu_i\}_{1\le i\le p(n+m+1)}$ the eigenvalues
of $R(j)$, we have, for any $\gamma>0$,
\begin{align*}
\prob\big\{z^* R(j) z > n \epsilon\big\}
&= \prob\Big\{\sum^{p(n+m+1)}_{i=1}\nu_i z^2_i > n \epsilon \Big\} \\
&\leq  e^{-n\gamma\eps}\, \prod_{i=1}^{p(n+m+1)}
\E\big\{e^{\gamma\nu_i z_i^2}\big\}\\
& =\exp\left(-n \Big(\gamma \epsilon + 
\frac{1}{2n} \sum^{(n+m+1)p}_{i=1} \log(1-2 \nu_i \gamma)  \Big)\right)\, .
\end{align*}
Let $\gamma = \frac{1}{2} (1-\sigma_{\max}) \epsilon$. Using the bound obtained for $|\max_i{\nu_i}|$  in Eq.~\eqref{eq:gradient_eigen_bound}, Lemma \ref{th:matrix_proper_grad}, $|2 \nu_i \gamma| \leq  \epsilon$. Now notice that if $|x|<1/2$ then $\log (1-x) > -x-x^2$. Thus, if we assume $\epsilon < 1/2$ and given that $\sum^{(n+m+1)p}_{i=1}\nu_i = 0$ (see Eq.~\eqref{eq:gradient_sum_eigen_bound}) we can continue the chain of inequalities,
\begin{align}
&\prob(\|\hG_S\|_\infty > \epsilon) \leq 2 |S| \max_{j} \exp \left(-n (\gamma \epsilon - 2 \gamma^2 \frac{1}{n} \sum^{(n+m+1)p}_{i=1} \nu^2_i)\right)\\
&\leq 2 |S| \exp \left(-n (\frac{1}{2} (1-\sigma_{\max})  \epsilon^2 - \frac{1}{4} (1-\sigma_{\max})^2 \epsilon^2  (1-\sigma_{\max})^{-1})\right)\\
&\leq 2 |S| \exp \left(-\frac{n}{4}  (1-\sigma_{\max}) \epsilon^2\right).
\end{align}
where the second inequality is obtained using the bound in Eq.~\eqref{eq:gradient_sum_eigen2_bound}.
\end{proofof}

\comment{
\begin{proposition}
Let $i,j\in \{ 1,...,p\}$,  $\sigma_{\max} \equiv \sigma_{max}(I + \eta A^0) < 1$, $T = \eta n > 3/D$ and $0 < \epsilon< 2/D$ where $D = (1-\sigma_{\max})/\eta$ then, 
\begin{eqnarray}
\prob(|\hQ_{ij} - Q^0_{ij})|  > \epsilon) \leq   2 e^{-\frac{n}{32 \eta^2}  (1-\sigma_{max})^3 \epsilon^2}.
\end{eqnarray} \label{en:bddQij}
\end{proposition}
}
\begin{proofof}{Proposition \ref{en:bddQij}}
The proof is very similar to that of proposition \ref{th:gradie_prob_bound}. We will first show that the bound
\begin{equation}
\prob(|\hQ_{ij} - \E(\hQ_{ij}) |  > \epsilon) \leq   2 e^{-\frac{n}{32 \eta^2}  (1-\sigma_{max})^3 \epsilon^2},
\end{equation}
holds in the case where the probability measure and expectation are taken with respect to trajectories $\{x(i)\}^n_{i=0}$ that started at time instant $t = -m$ with $x(-m) = w(-m)$. Assume we have done so. Now notice that as $m \rightarrow \infty$, $X$ converges in distribution to $n$ consecutive samples from the model \ref{eq:DiscreteTimeModel} when this is initiated from stationary state. In addition, as $m \rightarrow \infty$, we have from lemma \ref{th:EQhat_Q_start_bound} that $\E(\hQ_{ij}) \rightarrow Q^0_{ij}$. Since $\hQ_{ij}$ is a continuous function of $X = [x(0),...,x(n-1)]$, a simple application of the Continuous Mapping Theorem plus the fact that the upper bound is continuous in $\epsilon$ leads us to conclude that the bound also holds when the system is initiated from stationary state.

To prove our previous statement first recall the definition of $\hQ$ and notice that we can write,
\begin{equation}
\hQ_{ij}  =   \frac{\eta}{n} z^* R(i,j) z,
\end{equation}
where $z \in \reals^{m+n}$ is a vector of i.i.d. $\normal(0,1)$ and $R(i,j) \in \reals^{(n+m) \times (n+m)} $ is defined has in lemma \ref{th:matrix_proper_hess}. Letting $\nu_l$ denote the $l^{th}$ eigenvalue of the symmetric matrix $R(i,j)$ we can further write,
\begin{equation}
\hQ_{ij} - \E (\hQ_{ij}) = \frac{\eta}{n} \sum^{(n+m)p}_{l = 1} \nu_l (z^2_l - 1).
\end{equation}
By Lemma \ref{th:matrix_proper_hess} we know that,
\begin{align}
|\nu_l| &\leq \frac{1}{(1-\sigma_{\max})^2},\\
\frac{1}{n}\sum^{(n+m)p}_{l = 1} \nu^2_l &\leq \frac{2}{(1-\sigma_{\max})^3}\left(1 + \frac{3}{2n} \frac{1}{1-\sigma_{\max}} \right) \leq \frac{3}{(1-\sigma_{\max})^3},
\end{align}
where we applied $T > 3/D$ in the last line.

Now we are done since applying Bernstein trick, this time with $\gamma = 1/8 \, (1-\sigma_{\max})^3 \epsilon / \eta$, and making again use of the fact that $\log(1-x)> -x-x^2$ for $|x|<1/2$ we get,
\begin{align}
\prob ( \hQ_{ij} - \E (\hQ_{ij}) > \epsilon )  &= \prob( \sum^{(n+m)p}_{l = 1} \nu_l (z^2_l - 1) > \epsilon n / \eta ) \\
&\leq e^{-\frac{\gamma \epsilon n}{\eta}} e^{- \gamma \sum^{(n+m)p}_{l=1} \nu_l} +e^{-1/2 \sum^{(m+n)p}_{l=1} \log (1-2 \gamma \nu_l)}\\
&\leq e^{-\frac{\gamma \epsilon n}{\eta} - \gamma \sum^{(n+m)p}_{l=1} \nu_l + \gamma \sum^{(n+m)p}_{l=1} \nu_l +  2 \gamma^2 \sum^{(n+m)p}_{l=1} \nu^2_l}\\
&\leq e^{-\frac{n}{32 \eta^2} (1-\sigma_{\max})^3 \epsilon^2},
\end{align}
where had to assume that $\epsilon < 2/D$ in order to apply the bound on $\log(1-x)$.
An analogous reasoning leads us to,
\begin{equation}
\prob ( \hQ_{ij} - \E (\hQ_{ij}) < - \epsilon ) \leq e^{-\frac{n}{32 \eta^2} (1-\sigma_{\max})^3 \epsilon^2}
\end{equation}
and the results follows.

\end{proofof}

\begin{lemma}
As before, assume $\sigma_{\max} \equiv \sigma_{\max} (I + \eta A^0) < 1 $ and consider that model \eqref{eq:DiscreteTimeModel} was initiated at time $-m$ with $w(-m)$, that is, $x(-m) = w(-m)$ then
\begin{equation}
|\E(\hQ_{ij}) - Q^0_{ij}| \leq \frac{1}{n+m} \frac{\eta}{( 1 - \sigma_{\max})^2  }.
\end{equation} \label{th:EQhat_Q_start_bound}
\end{lemma}

\begin{proof}
Let $\rho = I + \eta A^0$. Since, 
\begin{equation}
Q^0_{ij} = \eta \sum^{\infty}_{l = 0} ({\rho}^l {\rho^*}^l)_{ij},
\end{equation}
and
\begin{equation}
\E(\hQ_{ij}) = \eta \sum^{n+m-1}_{l = 0} \frac{m+n-l}{n+m} (\rho^l {\rho^*}^l)_{ij},
\end{equation}
we can write, 
\begin{equation}
Q^0_{ij} - \E(\hQ_{ij}) = \eta \left(\sum^{\infty}_{l = m+n} ({\rho}^l {{\rho}^*}^l)_{ij} + \sum^{n+m-1}_{l = 1} \frac{l}{m+n} ({\rho}^l {{\rho}^*}^l)_{ij} \right).
\end{equation}
Using the fact that for any matrix $A$ and $B$ $\max_{ij}(A_{ij}) \leq \sigma_{\max} (A) $, $\sigma_{\max}(AB) \leq \sigma_{\max}(A)\sigma_{\max}(B)$ and $\sigma_{\max}(A + B) \leq \sigma_{\max}(A) + \sigma_{\max}(B)$ and introducing the notation $\zeta = \rho^2 $ we can write,
\begin{align}
|\E(\hQ_{ij}) - Q^0_{ij}| &\leq \eta \left( \frac{\zeta^{n+m}}{1-\zeta}  + \frac{\zeta}{n+m} \sum^{m+n-2}_{l = 0} \zeta^ l\right)= \frac{\eta (\zeta^2 + \zeta^{n+m} -2 \zeta^{m+n+1})}{(m+n)(1 - \zeta)^2} \\
&\leq  \frac{\eta}{(m+n)( 1 - \sigma_{max})^2  },
\end{align}
where we used the fact that for $\zeta \in [0,1]$ and $n \in \naturals$ we have $1-\zeta \geq 1-\sqrt{\zeta}$ and $\zeta^2 + \zeta^n - 2\zeta^{1+n} \leq 1$.
\end{proof}

\begin{proofof}{Theorem \ref{th:main_discrete}}

In order to prove Theorem \ref{th:main_discrete} we need to compute the probability that the conditions given by Proposition \ref{th:cond_to_hold} hold.
From the statement of the theorem we have that the first two conditions ($\alpha, C_{\min} > 0$) of Proposition \ref{th:cond_to_hold} hold. In order to make the first condition on $\hG$ imply the second condition on $\hG$ we assume that 
\begin{equation}
\frac{\lambda \alpha}{3} \leq \frac{A_{\min} C_{\min}}{4 k} - \lambda 	
\end{equation}
which is guaranteed to hold if
\begin{equation}
\lambda \leq A_{\min} C_{\min} / 8k. \label{eq:main_disc_lambda} 
\end{equation}
We also combine the two last conditions on $\hQ$ to
\begin{equation}
 \III \hQ_{[p], {S^0}} - Q^0_{[p], {S^0}} \III _\infty \leq \frac{\alpha}{12} \frac{C_{\min}}{\sqrt{k}}. 
\end{equation}
Where $[p]=S^0 \cup (S^0)^c$. We then impose that both the probability of the condition on $\hQ$ failing and the probability of the condition on $\hG$ failing are upper bounded by $\delta/2$. 
Using Proposition \ref{th:gradie_prob_bound} we see that the condition on $\hG$ fails with probability smaller than $\delta/2$ given that the following is satisfied
\begin{equation}
\lambda^2 = 36 \alpha^{-2} (n\eta D)^{-1} \log (4p/ \delta).
\end{equation}
But we also want \eqref{eq:main_disc_lambda} to be satisfied and so substituting $\lambda$ from the previous expression in \eqref{eq:main_disc_lambda} we conclude that $n$ must satisfy
\begin{equation}\label{eq:bdd_n1}
n \geq 2304 k^2 {C_{\min}}^{-2} {A_{\min}}^{-2} \alpha^{-2} (D \eta)^{-1} \log(4p/ \delta).
\end{equation}
In addition, the application of the probability bound in Proposition \ref{th:gradie_prob_bound} requires that
\begin{equation}
\frac{\lambda^2 \alpha^2}{9} < 1/4 
\end{equation}
so we need to impose further that,
\begin{equation}\label{eq:bdd_n2}
n \geq 16 (D \eta)^{-1} \log(4 p /\delta).
\end{equation}
To use Corollary \ref{th:prob_bound_inft_matrix_hess} for computing the probability that the condition on $\hQ$ holds we need,
\begin{equation}\label{eq:bdd_n3}
n \eta > 3/D,
\end{equation}
and
\begin{equation}
\frac{\alpha C_{\min}}{12 \sqrt{k}} <  2k D^{-1}.
\end{equation}
The last expression imposes the following conditions on $k$,
\begin{equation}
k^{3/2} > 24^{-1}  \alpha C_{\min} D \label{eq:main_th_restric_k}.
\end{equation}
The probability of the condition on $\hQ$ will be upper bounded by $\delta/2$ if  
\begin{equation}\label{eq:bdd_n4}
n > 4608 \eta^{-1} k^{3} \alpha^{-2} {C_{\min}}^{-2} D^{-3} \log {4pk / \delta}.
\end{equation}
The restriction \eqref{eq:main_th_restric_k} on $k$ looks unfortunate but since $k \geq 1$ we can actually show it always holds. Just notice $\alpha < 1$ and that
\begin{equation}
\sigma_{\max}(Q^0_{{S^0},{S^0}}) \leq \sigma_{\max}(Q^0) \leq \frac{\eta}{1-\sigma_{\max}} \Leftrightarrow D \leq \sigma_{\max}^{-1}(Q^0_{{S^0},{S^0}})
\end{equation}
therefore $C_{\min} D \leq \sigma_{\min}(Q^0_{{S^0},{S^0}}) / \sigma_{\max}(Q^0_{{S^0},{S^0}}) \leq 1$. This last expression also allows us to simplify the four restrictions on $n$ into a single one that dominates them. In fact, since $C_{\min} D \leq 1$ we also have $C_{\min}^{-2} D^{-2} \geq C_{\min}^{-1} D^{-1} \geq 1$ and this allows us to conclude that the only two conditions on $n$ that we actually need to impose are the one at Equations \eqref{eq:bdd_n1}, and \eqref{eq:bdd_n4}. A little more of algebra shows that these two inequalities are satisfied if

\begin{equation}
n \eta > \frac{10^4  k^2 (k D^{-2} + A_{\min}^{-2})}{\alpha^{2} D C_{\min}^{2} }  \log(4p k/ \delta).
\end{equation}
This conclude the proof of Theorem \ref{th:main_discrete}.

\end{proofof}

\begin{lemma}
Let $\sigma_{\max} \equiv \sigma_{\max} (I + \eta A^0)$ and $\rho_{\min}(A^0) = - \lambda_{\max}(( A^0 + (A^0)^*) /2) > 0$ then,
\begin{align}
-\lambda_{\min}\left(\frac{A^0 + (A^0)^*}{2}\right) &\geq \limsup_{\eta \rightarrow 0} \frac{1 - \sigma_{\max}}{\eta},\\
\liminf_{\eta \rightarrow 0} \frac{1 - \sigma_{\max}}{\eta} &\geq -\lambda_{\max}\left(\frac{A^0 + (A^0)^*}{2}\right).
 \end{align} \label{th:limit_D}
\end{lemma}
\begin{proof}
 \begin{align}
  \frac{1 - \sigma_{\max}}{\eta} &= \frac{1 -  \lambda^{1/2}_{\max}((I + \eta A^0)^* (I + \eta A^0))}{\eta} \\
  &= \frac{1 -  \lambda^{1/2}_{max}(I + \eta (A^0 + (A^0)^*) + \eta^2 (A^0)^* A^0 ) }{\eta} \\
  &= \frac{1 - (1 + \eta u^* (A^0 + (A^0)^* + \eta (A^0)^* A^0 ) u)^{1/2}}{\eta},
 \end{align}
where $u$ is some unit vector that depends on $\eta$. Thus, since $\sqrt{1+x} = 1 + x/2 + O(x^2)$,
\begin{equation}
\liminf_{\eta \rightarrow 0} \frac{1 - \sigma_{\max}}{\eta} = - \limsup_{\eta \rightarrow 0} u^*\left( \frac{A^0 + (A^0)^*}{2} \right) u \geq -\lambda_{\max}\left(\frac{A^0 + (A^0)^*}{2} \right).
\end{equation}
The other inequality is proved in a similar way.
\end{proof}

\begin{proofof}{Theorem \ref{th:cont_reg_graph_bounded_degree}}

In order to prove Theorem \ref{th:cont_reg_graph_bounded_degree} we first state and prove the following lemma,
\begin{lemma}
Let $G$ be a simple connected graph of vertex degree bounded above by $k$. Let $\tilde{A}$ be its adjacency matrix and $A^0 = -h I + \tilde{A}$ with $h > k$ then for this $A^0$ the system in \eqref{eq:BasicModel} has $Q^0 = -(1/2) (A^0)^{-1}$ and,
\begin{equation}
 \III Q^0_{(S^0)^C, {S^0}} (Q^0_{{S^0},{S^0}})^{-1}  \III _{\infty} =  \III (A^0_{(S^0)^C, (S^0)^C})^{-1} A^0_{(S^0)^C, S^0} \III _{\infty} \leq k/h.
\end{equation}
\label{th:cont_time_incoher_bound}
\end{lemma}
\begin{proof}
$\tilde{A}$ is symmetric so $A^0$ is symmetric. Since $\tilde{A}$ is irreducible and non-negative, Perron-Frobenious theorem tells that $\lambda_{\max}(\tilde{A}) \leq k$ and consequently $\lambda_{\max}(A^0) \leq -h + \lambda_{\max}(\tilde{A}) \leq -h + k$. Thus $h > k$ implies that $A^0$ is negative definite and using equation \eqref{eq:Lyapunov} we can compute  $Q^0 = -(1/2) (A^0)^{-1}$. Now notice that, by the block matrix inverse formula, we have
\begin{align}
(Q^0_{{S^0},{S^0}})^{-1} &=  -2 C^{-1},\\
Q^0_{(S^0)^C, S^0} &= \frac{1}{2}( (A^0_{(S^0)^C, (S^0)^C})^{-1} A^0_{(S^0)^C, {S^0}} C),
\end{align}
where $C = A^0_{{S^0},{S^0}} - A^0_{{S^0}, (S^0)^C} (A^0_{(S^0)^C, (S^0)^C})^{-1} A^0_{(S^0)^C, {S^0}}$ and thus
\begin{equation}
 \III Q^0_{(S^0)^C, {S^0}} (Q^0_{{S^0},{S^0}})^{-1}  \III _{\infty} =  \III (A^0_{(S^0)^C, (S^0)^C})^{-1} A^0_{(S^0)^C, {S^0}} \III _{\infty}.
\end{equation}
Recall the definition of $ \III B \III _\infty$,
\begin{equation}
 \III B \III _\infty = \max_i \sum_j |B_{ij}|.\label{eq:infty_mat_norm}
\end{equation}
Let $z = h^{-1}$ and write,
\begin{align}
(A^0_{(S^0)^C, (S^0)^C})^{-1} &= -z (I - z \tilde{A}_{(S^0)^C, (S^0)^C})^{-1} =  -z \sum^\infty_{n = 0} (z \tilde{A}_{(S^0)^C, (S^0)^C})^n,\\
A^0_{(S^0)^C, S^0} &= z^{-1} z \tilde{A}_{(S^0)^C, {S^0}}.
\end{align}
This allows us to conclude that $ \III (A^0_{(S^0)^C, (S^0)^C})^{-1} A^0_{(S^0)^C, {S^0}} \III _{\infty}$ is in fact the maximum over all path generating functions of paths starting from a node $i \in (S^0)^C$ and hitting  ${S^0}$ for a first time. Let $\Omega_i$ denote this set of paths, $\omega$ a general path in $G$ and $|\omega|$ its length. Let $k_1, ..., k_{|\omega|}$ denote the degree of each vertex visited by $\omega$ and note that $k_m \leq k, \forall m$. Then each of these path generating functions can be written in the following form,
\begin{equation}
\sum_{\omega \in \Omega_i } z^{|\omega|}  \leq \sum_{\omega \in \Omega_i } \frac{1}{k_1...k_{|\omega|}}(kz)^{|\omega|} = \E_G ((kz)^{T_{i,{S^0}}}),\label{eq:inc_cond_path}
\end{equation}
where $T_{i,{S^0}}$ is the first hitting time of the set ${S^0}$ by a random walk that starts at node $i \in {S^0}^C$ and moves with equal probability to each neighboring node. But $T_{i,{S^0}} \geq 1$ and $kz < 1$ so the previous expression is upper bounded by $kz$.
\end{proof}

Now what remains to complete the proof of Theorem \ref{th:cont_reg_graph_bounded_degree} is to compute the quantities $\alpha$, $A_{\min}$, $\rho_{\min} (A^0)$ and $C_{\min}$ in Theorem \ref{th:main_cont_time} . From Lemma \ref{th:cont_time_incoher_bound} we know that $\alpha = 1 - k/(k+m)$. Clearly,  $A_{\min} = 1$. We also have that $\rho_{\min}(A^0) = \sigma_{\min} (A^0) \geq k+m - \sigma_{\max} (\tilde{A}) \geq m+k - k = m$. Finally,
\begin{equation}
\lambda_{\min}(Q^0_{{S^0},{S^0}}) = \frac{1}{2}\lambda_{\min}(-(A^0)^{-1}) = \frac{1}{2} \frac{1}{\lambda_{\max}(-A^0)} \geq \frac{1}{2} \frac{1}{m+k+k} \geq \frac{1}{4(m+k)}
\end{equation}
where in the last step we made use of the fact that $m+k > k$. Substituting these values in the inequality from Theorem  \ref{th:main_cont_time} gives the desired result.

\end{proofof}

\newpage

\end{document}